\documentclass[11pt,a4paper]{article}
\usepackage[cp1251]{inputenc}
\usepackage{amsmath,amsthm}
\usepackage{amsfonts,amstext,amssymb,verbatim,epsfig}
\usepackage{dsfont}
\usepackage{enumerate}

\def\R{{\mathbb{R}}}

\def\Z{{\mathbb{Z}}}

\newcommand{\EE}{\mathbb{E}}

\newcommand{\GG}{\mathbb{G}}

\newcommand{\ld}{l(d)} 

\def\berbond{{\widetilde{\mathcal{B}}}} 
\def\intgraph{{\widetilde{\mathcal{I}}}} 
\def\gengraph{{\widetilde{J}}} 
\def\gengraphrand{{\widetilde{\mathcal{J}}}} 
\def\edge{{\widetilde{e}}}
\def\badseed{{\overline G}}

\newtheorem{theorem}{Theorem}
\newtheorem{corollary}{Corollary}
\newtheorem{lemma}{Lemma}

\newtheoremstyle{likedef}
  {}%
  {}%
  {}%
  {\parindent}%
  {\bfseries}%
  {.}%
  {.5em}%
  {}%

\theoremstyle{likedef}
\newtheorem{definition}{Definition}[section]
\newtheorem{remark}{Remark}

\numberwithin{equation}{section}

\begin{document}

\title{The effect of small quenched noise on connectivity properties\\ of random interlacements}

\author{Bal\'azs R\'ath\thanks{ETH Z\"urich, Department of Mathematics, R\"amistrasse 101, 8092 Z\"urich. Email: balazs.rath@math.ethz.ch and artem.sapozhnikov@math.ethz.ch.
The research of both authors has been supported by the grant ERC-2009-AdG  245728-RWPERCRI.}
\and
Art\"{e}m Sapozhnikov\footnotemark[1]
}

\maketitle


\begin{abstract}
Random interlacements (at level $u$) is a one parameter 
family of random subsets of $\Z^d$ introduced by Sznitman in \cite{SznitmanAM}.
The vacant set at level $u$ is the complement of the random interlacement at level $u$. 
While the random interlacement induces a connected subgraph of $\Z^d$ for all levels $u$, the 
vacant set has a non-trivial phase transition in $u$, as shown in \cite{SznitmanAM} and \cite{Sznitman_Sidoravicius_cpam}. 

In this paper, we study the effect of small quenched noise on connectivity properties of 
the random interlacement and the vacant set.
For a positive $\varepsilon$, we allow each vertex of the random interlacement (referred to as occupied) to become vacant, 
and each vertex of the vacant set to become occupied with probability $\varepsilon$, 
independently of the randomness of the interlacement, and independently for different vertices.
We prove that for any $d\geq 3$ and $u>0$, almost surely, 
the perturbed random interlacement percolates for small enough noise parameter $\varepsilon$. 
In fact, we prove the stronger statement that Bernoulli percolation on the random interlacement graph 
has a non-trivial phase transition in wide enough slabs. 
As a byproduct, we show that any electric network with i.i.d. positive resistances on the interlacement graph is transient, 
which strengthens our result in \cite{RS:Transience}. 
As for the vacant set, we show that for any $d\geq 3$, there is still a non-trivial phase transition in $u$ 
when the noise parameter $\varepsilon$ is small enough, and we give explicit upper and lower bounds on the 
value of the critical threshold, when $\varepsilon\to 0$.  
\end{abstract}

\section{Introduction}
\noindent

The model of random interlacements was recently introduced by Sznitman in \cite{SznitmanAM} 
in order to describe the local picture left by the trajectory of a random walk on
the discrete torus $(\Z/N\Z)^d$, $d \geq 3$ when it runs up to times of order $N^d$,
or on the discrete cylinder $(\Z/N\Z)^d \times \Z$ , $d \geq 2$, 
when it runs up to times of order $N^{2d}$, see \cite{sznitman_cylinders}, \cite{windisch_torus}. 
Informally, the random interlacement Poisson point process consists of a countable collection of doubly infinite trajectories on $\Z^d$, 
and the trace left by these trajectories on a finite subset of $\Z^d$ ``looks like'' the trace of the above mentioned random walks. 

The set of vertices visited by at least one of these trajectories is the random interlacement at 
level $u$ of Sznitman \cite{SznitmanAM}, and the complement of this set is the vacant set at level $u$. 
These are one parameter families of translation invariant, ergodic, long-range correlated random subsets of $\Z^d$, 
see \cite{SznitmanAM}. 
We call the vertices of the random interlacement occupied, and the vertices of the vacant set vacant. 
While the set of occupied vertices induces a connected subgraph of $\Z^d$ for all levels $u$, the 
graph induced by the set of vacant vertices has a non-trivial phase transition in $u$, as shown in \cite{SznitmanAM} and \cite{Sznitman_Sidoravicius_cpam}. 

The effect of introducing a small amount of quenched disorder into a system with long-range correlations 
on the phase transition has got a lot of attention (see, e.g., \cite{Halperin_Weinrib}, \cite{Weinrib}, \cite{CCFS86}, \cite{CCFS89}). 
In this paper we consider how small quenched disorder affects the connectivity properties of the random interlacement and the vacant set. 
For $\varepsilon>0$, given a realization of the random interlacement, 
we allow each vertex independently to switch from occupied to vacant and from vacant to occupied with probability $\varepsilon$, 
and we study the effect it has on the existence of an infinite connected component in the graphs of occupied or vacant vertices. 

We prove that for any $d\geq 3$ and $u>0$, almost surely, the set of occupied vertices percolates for small enough noise parameter $\varepsilon$. 
In fact, we prove the stronger statement that Bernoulli percolation on the random interlacement graph 
has a non-trivial phase transition in wide enough slabs. 
The two main ingredients of our proof are a strong connectivity lemma for the interlacement graph proved in
 \cite{RS:Transience} and Sznitman's decoupling inequalities from \cite{Sznitman:Decoupling}.
As a byproduct, we show that any electric network with i.i.d. positive resistances on the interlacement graph is transient, 
which strengthens our result in \cite{RS:Transience}. 

We also prove that for any $d\geq 3$, 
the set of vacant vertices still undergoes a non-trivial phase transition in $u$ when the noise parameter $\varepsilon$ is small enough, 
and give explicit upper and lower bounds on the value of the threshold, when $\varepsilon\to 0$. 
The bounds that we derive suggest that the vacant set phase transition is robust with respect to noise, 
which we state as a conjecture.

\subsection{The model}\label{sec:model}
\noindent

For $x\in\Z^d$, $d\geq 3$, let $P_x$ be the law of a simple random walk $X$ on $\Z^d$ with $X(0) = x$. 
Let $K$ be a finite subset of $\Z^d$. The equilibrium measure of $K$ is defined by 
\[
 e_K(x) = P_x\left[X(t)\notin K~\mbox{for all}~t\geq 1\right] ,\quad \mbox{for }x\in K ,\
\]
and $e_K(x) = 0$ for $x\notin K$. 
The capacity of $K$ is the total mass of the equilibrium measure of $K$:
\[
 \mathrm{cap}(K) = \sum_x e_K(x) .\
\]
Since $d\geq 3$, for any finite set $K\subset \Z^d$, the capacity of $K$ is positive. 
Therefore, we can define the normalized equilibrium measure by 
\[
 \widetilde e_K(x) = e_K(x)/\mathrm{cap}(K) .\
\]

Let $W$ be the space of doubly-infinite nearest-neighbor trajectories in $\Z^d$ ($d\geq 3$) which tend to infinity at positive and negative infinite times, and 
let $W^*$ be the space of equivalence classes of trajectories in $W$ modulo time-shift. 
We write $\mathcal W$ for the canonical $\sigma$-algebra on $W$ generated by the coordinate maps, and 
$\mathcal W^*$ for the largest $\sigma$-algebra on $W^*$ for which the canonical map $\pi^*$ from $(W,\mathcal W)$ to $(W^*,\mathcal W^*)$ is measurable. 

Let $\mu$ be a Poisson point measure on $W^*$. 
For a finite subset $K$ of $\Z^d$, denote by $\mu_K$ the restriction of $\mu$ to the set of trajectories from $W^*$ that intersect $K$, and 
by $N_K$ be the number of trajectories in $\mathrm{Supp}(\mu_K)$. 
The point measure $\mu_K$ can be written as $\mu_K = \sum_{i=1}^{N_K}\delta_{\pi^*(X_i)}$, 
where $X_i$ are doubly-infinite trajectories from $W$ parametrized in such a way that 
$X_i(0) \in K$ and $X_i(t) \notin K$ for all $t<0$ and for all $i\in\{1,\ldots,N_K\}$. 

For $u>0$, we say that a Poisson point measure $\mu$ on $W^*$ has distribution $\mathrm{Pois}(u,W^*)$ if the following properties hold: 
\begin{itemize}
\item[(1)]\label{distr:1}
The random variable $N_K$ has Poisson distribution with parameter $u\mathrm{cap}(K)$. 
\item[(2)]\label{distr:2}
Given $N_K$, the points $X_i(0)$, $i\in\{1,\ldots,N_K\}$, are independent and distributed according to the normalized equilibrium measure on $K$. 
\item[(3)]\label{distr:3}
Given $N_K$ and $(X_i(0))_{i=1}^{N_K}$, the corresponding forward and backward paths are conditionally independent, $(X_i(t), t\geq 0)_{i=1}^{N_K}$
 are distributed as independent
 simple random walks, and $(X_i(t), t\leq 0)_{i=1}^{N_K}$ are distributed as independent random walks conditioned on not hitting $K$.
\end{itemize}
Properties (1)-(3) uniquely define $\mathrm{Pois}(u,W^*)$, as proved in Theorem~1.1 in \cite{SznitmanAM}. 
In fact, Theorem~1.1 in \cite{SznitmanAM} gives a coupling of 
the Poisson point measures $\mu(u)$ with distribution $\mathrm{Pois}(u,W^*)$ for all $u>0$.
We refer the reader to \cite{SznitmanAM} for more details. 

Let $\EE^d$ be the set of edges of $\Z^d$, i.e., $\EE^d = \{ \{x,y\}~:~x,y\in\Z^d, |x-y|_1 = 1\}$. 
We will use the following convention throughout the paper. 
For a subset $J$ of $\EE^d$, the subgraph of the lattice $(\Z^d,\EE^d)$ with the vertex set $\Z^d$ and the edge set $J$ will be also denoted by $J$. 

For a Poisson point measure $\mu$ with distribution $\mathrm{Pois}(u,W^*)$, 
the {\it random interlacement } $\mathcal I^u = \mathcal I^u(\mu)$ (at level $u$) is defined in \cite{SznitmanAM} as the set of vertices of $\Z^d$ 
visited by at least one of the trajectories from $\mathrm{Supp}(\mu)$. 
This is a translation invariant and ergodic random subset of $\Z^d$, as shown in \cite[Theorem~2.1]{SznitmanAM}. 
The law of $\mathcal I^u$ is characterized by the identity (see (0.10) and Remark~2.2 (2) in \cite{SznitmanAM}):
\[
\mathbb P\left[\mathcal I^u\cap K = \emptyset\right] = e^{-u \mathrm{cap}(K)} , \quad \text{for all finite $K\subseteq \Z^d$} .\
\]
We denote by $\intgraph^u = \intgraph^u(\mu)$ the set of edges of $\EE^d$ traversed by at least one of the trajectories from $\mathrm{Supp}(\mu)$. 
The corresponding random subgraph $\intgraph^u$ of $(\Z^d,\EE^d)$ (with the vertex set $\Z^d$ and the edge set $\intgraph^u$) 
is called the {\it random interlacement graph} (at level $u$). 
It follows from Theorem~2.1 and Remark~2.2(4) of \cite{SznitmanAM}  
that $\intgraph^u$ is a translation invariant ergodic random subgraph of $(\Z^d,\EE^d)$. 
Let $\mathcal V^u = \Z^d\setminus\mathcal I^u$ be the {\it vacant set} at level $u$. 

\medskip

Given a parameter $\varepsilon\in(0,1)$, we consider the family $\theta_x$, $x\in\Z^d$, 
of independent Bernoulli random variables (an independent noise) with parameter $\varepsilon$, and define 
$\varepsilon$-disordered analogues of the random interlacement $\mathcal I^{u,\varepsilon}$ and the vacant set $\mathcal V^{u,\varepsilon}$ as follows. 
We say that $x\in \mathcal I^{u,\varepsilon}$ if $x\in\mathcal I^u$ and $\theta_x = 0$ or $x\in \mathcal V^u$ and $\theta_x = 1$. 
In other words, the vertices of the random interlacement get an $\varepsilon$-chance to become vacant, and the vertices of the vacant set 
get an $\varepsilon$-chance to become occupied. Let $\mathcal V^{u,\varepsilon} = \Z^d\setminus\mathcal I^{u,\varepsilon}$. 
We are interested in percolative properties of $\mathcal I^{u,\varepsilon}$ and $\mathcal V^{u,\varepsilon}$. 
It follows from Remark~1.6(4) in \cite{SznitmanAM} that for any $d\geq 3$ and $u>0$,
\[
\mbox{cov}_u\left[\mathds{1}(x\in\mathcal V^u), \mathds{1}(y\in\mathcal V^u)\right] \asymp 
\left(1+|x-y|_\infty\right)^{2-d},\quad \mbox{for }x,y\in\Z^d ,\
\]
where $\mbox{cov}_u$ denotes the covariance under $\mathrm{Pois}(u,W^*)$. This displays the presence of long-range correlations in $\mathcal V^u$. 
Non-rigorous study of the effect of small quenched noise on the critical behavior of a system with long-range correlations 
was initiated in \cite{Halperin_Weinrib, Weinrib}. 

It was shown, among other results, in \cite{SznitmanAM} 
that the random interlacement graph $\intgraph^u$ consists of a unique infinite connected component and isolated vertices.
(Refinements of this result were obtained in \cite{LT,PT,RS}.)
In \cite{RS:Transience}, we showed that the random interlacement graph is almost surely transient for any $u>0$ in dimensions $d\geq 3$.
In Theorem~\ref{thm:slabs} of the present paper, we prove that for any $u>0$ and small enough $\varepsilon>0$, 
the set $\mathcal I^{u,\varepsilon}$ still contains an infinite connected component. 
In fact, Theorem~\ref{thm:slabs} implies that 
$\mathcal I^u$ and $\intgraph^u$ still have an infinite connected component in wide enough slabs, 
even after a small positive density of vertices of $\mathcal I^u$, respectively edges of $\intgraph^u$, 
is removed. 
One might interpret all these results as an evidence of the heuristic statement that 
the geometry of the interlacement graph is similar to that
of the underlying lattice $\Z^d$. 
Recently, this question has been settled in \cite{CP} by a clever refinement of the techniques in \cite{RS,RS:Transience}. 
It was proved in \cite{CP} (and later in \cite{DRS:ChemDist} with a different, model independent proof) that the graph distance in $\mathcal I^u$ 
is comparable to the graph distance in $\Z^d$, and a shape theorem holds for 
balls with respect to graph distance on $\mathcal I^u$. 
First results about heat-kernel bounds for the random walk on $\mathcal I^u$ have been recently obtained in \cite[Theorem~2.3]{PrShel}.

\medskip

An important role in understanding 
the local picture left by the trajectory of a random walk on
the discrete torus $(\Z/N\Z)^d$, $d \geq 3$ or the discrete cylinder $(\Z/N\Z)^d \times \Z$ , $d \geq 2$ 
is played by  
\[
u_*= \inf \{ u \geq 0 \; : \; \mathbb{P} [ 0 \leftrightarrow \infty \mbox{ in }\mathcal{V}^u]=0  \}  
\]
(see, e.g., \cite{Sznitman:Upperbound,teixeira_windisch}).
It follows from \cite[(1.53) and (1.55)]{SznitmanAM} that for $u<u'$, 
the set $\mathcal V^{u'}$ is stochastically dominated by $\mathcal V^u$. 
Therefore, for all $u>u_*$, 
$\mathbb{P} [ 0 \leftrightarrow \infty \mbox{ in }\mathcal{V}^u]=0$. 
Moreover, by \cite{Sznitman_Sidoravicius_cpam,SznitmanAM}, $u_*\in(0,\infty)$, i.e., 
there is a non-trivial phase transition for $\mathcal V^u$ in $u$ at $u_*$. 
In Theorem~\ref{thm:Vuvarepsilon} of this paper, we prove that for small enough $\varepsilon$, 
the $\varepsilon$-disordered vacant set $\mathcal V^{u,\varepsilon}$ 
still undergoes a non-trivial phase transition in $u$. 
In Theorem~\ref{thm:u*varepsilon} we give explicit upper and lower bounds on the phase transition threshold for 
$\mathcal V^{u,\varepsilon}$, as $\varepsilon\to 0$. 
These bounds suggest that the phase transition is actually robust with respect to noise. 
We state it as a conjecture in Remark~\ref{rem:conjecture:u*varepsilon}.

\section{Main results}\label{sec:mainresults}
\noindent

For $p\in(0,1)$, we define the random subset $\berbond^p$ of $\mathbb{E}^d$ by 
deleting each edge  with probability $(1-p)$ and retaining it with probability $p$, independently for all edges, 
and, similarly, the random subset $\mathcal B^p$ of $\Z^d$ by deleting every vertex of $\Z^d$ with probability $(1-p)$ and retaining it with probability $p$, 
independently for all vertices. 
We look at the random subgraphs of $(\Z^d,\mathbb{E}^d)$ with vertex set $\Z^d$  and edge set $\intgraph^u \cap \berbond^p$, 
and the one induced by the set of vertices $\mathcal I^u\cap\mathcal B^p\subset \Z^d$. 

\medskip

Our first theorem states that the graphs $\mathcal I^u$ and $\intgraph^u$ have infinite connected subgraphs in a wide enough slab, 
moreover, Bernoulli bond percolation on $\intgraph^u$ and Bernoulli site percolation on $\mathcal I^u$ 
restricted to this slab have a non-trivial phase transition. 

\begin{theorem}\label{thm:slabs}
Let $d\geq 3$ and $u>0$. There exist $p<1$ and $R \geq 1$ such that, almost surely, 
the random graphs $\mathcal I^u\cap\mathcal B^p$ and $\intgraph^u \cap \berbond^p$ contain 
infinite connected components in the slab $\Z^2\times[0,R)^{d-2}$. 
\end{theorem}
As a byproduct of the proof of Theorem~\ref{thm:slabs}, we obtain 
the following generalization of the main result in \cite{RS:Transience}. 

\begin{theorem}\label{thm:transience}
Let $d\geq 3$ and $u>0$. 
Let $R_{\edge}$, $\edge \in \mathbb{E}^d$ be independent identically distributed positive random variables. 
The electric network $\{\edge~:~\edge \in \intgraph^u\}$ with resistances $R_{\edge}$ is almost surely transient, 
i.e., the effective resistance between any vertex in $\intgraph^u$ and infinity is finite. 
\end{theorem}
Theorem \ref{thm:transience} is a generalization of the main result of \cite{RS:Transience},
since the transience of the unique infinite connected component of the 
random interlacement graph $\intgraph^u$ follows from the case when $R_{\edge}$ are almost surely equal to $1$ 
(see, e.g., \cite{DS}). 
The result of Theorem~\ref{thm:transience} is equivalent (see the main result of \cite{PP:Transience}) to the following statement: 
for any $u>0$, there exists $p<1$ such that the graph $\intgraph^u \cap \berbond^p$ contains 
a transient component, i.e., the simple random walk on it is transient. 
The proof of this fact will come as a byproduct of the proof of Theorem~\ref{thm:slabs}.

The main idea of the proofs of Theorems~\ref{thm:slabs} and \ref{thm:transience} is renormalization. 
We partition the graph $\Z^d$ into disjoint blocks of equal size. 
A block is called good if the graph $\intgraph^u$ contains a unique large connected component in this block 
and all the edges of the block are in $\berbond^p$, otherwise it is called bad. 
A more precise definition will be given in Section~\ref{sec:connectivity}. 
It will be shown that paths of good blocks contain paths of $\intgraph^u\cap\berbond^p$.
In particular, percolation of good blocks implies percolation of $\intgraph^u\cap\berbond^p$. 
Using the strong connectivity result of \cite{RS:Transience}, stated as Lemma~\ref{l:connectivity} below, 
we show that a block is good with probability tending to $1$, as the size of the block increases. 
We then use the decoupling inequalities of \cite{Sznitman:Decoupling}, stated as Theorem~\ref{thm:decoupling} below, 
to show in Lemma~\ref{l:badpaths} that $*$-connected components of bad blocks are small. 
With the result of Lemma~\ref{l:badpaths}, the existence statement of Theorem~\ref{thm:slabs} follows 
using a standard duality argument, and 
the proof of Theorem~\ref{thm:transience} is reminiscent of the proof of Theorem~1 in Section~3 of \cite{RS:Transience}. 

\medskip

In our next theorem, we show that for small enough $\varepsilon>0$, 
the $\varepsilon$-disordered vacant set $\mathcal V^{u,\varepsilon}$ undergoes a non-trivial phase transition in $u$. 
Let 
\[
u_*(\varepsilon)= \inf \{ u \geq 0 \; : \; \mathbb{P} [ 0 \leftrightarrow \infty \mbox{ in }\mathcal{V}^{u,\varepsilon}]=0  \} .\
\]
\begin{theorem}\label{thm:Vuvarepsilon}
Let $d\geq 3$. For any $\varepsilon\in(0,1/2)$ and $u>u_*(\varepsilon)$, 
\[
\mathbb{P} [ 0 \leftrightarrow \infty \mbox{ in }\mathcal{V}^{u,\varepsilon}]=0 .\
\]
In other words, for $\varepsilon\in(0,1/2)$, 
the $\varepsilon$-disordered vacant set $\mathcal V^{u,\varepsilon}$ undergoes 
a phase transition in $u$ at $u_*(\varepsilon)$. 
Moreover, there exists $\varepsilon_0>0$ such that for all $\varepsilon\leq \varepsilon_0$, 
\[
0<  u_*(\varepsilon) < \infty .\
\]
\end{theorem}
The first statement of Theorem~\ref{thm:Vuvarepsilon} is proved in Lemma~\ref{l:coupling}. 
It follows from a standard coupling argument and 
the fact that the set $\mathcal V^{u'}$ is stochastically dominated by $\mathcal V^{u}$ for $u<u'$
(see \cite[(1.53) and (1.55)]{SznitmanAM}). 
The second statement of Theorem~\ref{thm:Vuvarepsilon} follows 
from the more general statement of Theorem~\ref{thm:u*varepsilon}, in which we 
give explicit upper and lower bounds on $u_*(\varepsilon)$, as $\varepsilon\to 0$. 
The proof of Theorem~\ref{thm:u*varepsilon} uses renormalization, 
and is very similar in spirit to the proof of Theorem~\ref{thm:slabs}. 

\medskip

The bounds on $u_*(\varepsilon)$ that we obtain in Theorem~\ref{thm:u*varepsilon} 
are in terms of certain thresholds describing local behavior of $\mathcal V^u$ in sub- and supercritical regimes
(see \eqref{def:u**} and Definition~\ref{def:tildeu}, respectively). 
In particular, they are purely in terms of $\mathcal V^u$ and not $\mathcal V^{u,\varepsilon}$. 
As we discuss in Remark~\ref{rem:conjecture:u*varepsilon}, 
these thresholds are conjectured to coincide with $u_*$, 
therefore it is reasonable to believe that 
the phase transition of $\mathcal V^u$ is stable with respect to small random noise. 
In other words, the following conjecture holds: 
\[
 \lim_{\varepsilon\to 0} u_*(\varepsilon) = u_* .\
\]
Finally, note that it is essential for $u_*(\varepsilon)<\infty$ that the parameter $\varepsilon$ is small. 
For example, since $\mathcal V^{u,1/2}$ has the same law as the Bernoulli site percolation with parameter $1/2$, 
which is supercritical in dimensions $d\geq 3$ (see \cite{CampaninoRusso}), 
we obtain that $u_*(1/2) = \infty$. 

\medskip

We now describe the structure of the remaining sections of the paper. 
We recall the strong connectivity lemma of \cite{RS:Transience} and 
the decoupling inequalities of \cite{Sznitman:Decoupling} in Section~\ref{section:notationandproperties}. 
In Section~\ref{subsection:seed} we construct and study seed events which are used in Section~\ref{sec:connectivity} to define good blocks. 
Lemma~\ref{l:badpaths}, the main ingredient of the proofs of Theorems~\ref{thm:slabs} and \ref{thm:transience}, is proved in Section~\ref{sec:connectivity}. 
The proofs of Theorems \ref{thm:slabs} and \ref{thm:transience} are given in Section~\ref{sec:proofs}, 
and the proof of Theorem~\ref{thm:Vuvarepsilon} is given in Section~\ref{sec:vacantset}, 
where we also give explicit bounds on $u_*(\varepsilon)$, as $\varepsilon \to 0$. 

\section{Notation and known results}\label{section:notationandproperties}
\noindent

In this section we introduce basic notation and collect some properties of the random interlacements, 
which are recurrently used in our proofs.

\subsection{Notation}
\noindent

For $a\in\R$, we write $|a|$ for the absolute value of $a$, and $\lfloor a\rfloor$ for the integer part of $a$.
For $(x_1,\dots,x_d)= x\in \Z^d$, we write $|x|_\infty$ for the $l^\infty$-norm of $x$, i.e., $|x|_\infty = \max\left(|x_1|,\ldots,|x_d|\right)$, 
and $|x|_1$ for the $l^1$-norm of $x$, i.e., $|x|_1 = \sum_{i=1}^d|x_i|$. 
For $R>0$ and $x\in\Z^d$, let $B(x,R) = \{y\in\Z^d~:~|x-y|_\infty\leq R\}$  be the $l^\infty$-ball of radius $R$ centered at $x$, and $B(R) = B(0,R)$.

For $x\in\Z^d$ and integers $m<n$, we write $x+[m,n)^d$ for the set of vertices  
$y = (y_1,\ldots,y_d)\in\Z^d$ with $m\leq y_i-x_i < n$ for all $i\in\{1,\ldots,d\}$. 
For $\edge \in\EE^d$, we write $\edge \in x+[m,n)^d$ if both of its endvertices are in $x+[m,n)^d$. 
If $\gengraph \subseteq \EE^d$, we denote by $\gengraph \cap (x+[m,n)^d)$ the set of edges of $\gengraph$ with 
both endvertices in $x+[m,n)^d$. 
For $x,y\in\Z^d$, we write $x\leftrightarrow y$ in $\gengraph$, 
if $x$ and $y$ are in the same connected component of the graph $\gengraph$. 

Let $(\Omega_1,\mathcal{F}_1,\mathbb{P}^u)$, with $\Omega_1 = \{0,1\}^{\mathbb E^d}$ and the canonical $\sigma$-algebra $\mathcal F_1$, be 
the probability space on which $\intgraph^u$ is defined. 
For $\omega\in \Omega_1$, we say that $\edge\in\mathbb E^d$ is in $\intgraph^u$ when $\omega_\edge = 1$. 
Let $(\Omega_2,\mathcal{F}_2,\mathbf{P}_p)$, with $\Omega_2 = \{0,1\}^{\mathbb E^d}$ and the canonical $\sigma$-algebra $\mathcal F_2$, be 
the probability space on which $\berbond^p$ is defined. 
For $\omega\in \Omega_2$, we say that $\edge\in\mathbb E^d$ is in $\berbond^p$ when $\omega_\edge = 1$. 
Finally, let $(\Omega,\mathcal{F},\mathbb{P}) = (\Omega_1\times\Omega_2,\mathcal{F}_1\times\mathcal{F}_2,\mathbb{P}^u\otimes\mathbf{P}_p)$ 
denote the probability space on which the random interlacement graph $\intgraph^u$ and Bernoulli bond 
percolation configuration $\berbond^p$ are jointly defined. 

Throughout the paper, we use the following notational agreement. 
For events $A_1\in\mathcal F_1$ and $A_2\in\mathcal F_2$, we denote the corresponding events $A_1\times\Omega_2$ and $\Omega_1\times A_2$ in $\mathcal F$ 
also by $A_1$ and $A_2$, respectively. 
We denote by $\mathds{1}(A)$ the indicator of event $A$ and by $A^c$ the complement of $A$.
For $i\in\{1,2\}$, given a random subset $\gengraphrand(\omega)$ of $\EE^d$, with $\omega\in\Omega_i$, 
and an event $A\in\mathcal F_i$, we define 
\begin{equation}\label{eq:eventsubset}
 A(\gengraphrand) = \{\omega\in\Omega_i~:~\chi_{\gengraphrand (\omega)}\in A\} ,\
\end{equation}
where for $\edge \in\EE^d$, $\chi_{\gengraphrand (\omega)}(\edge)$ equals $1$ if $\edge\in\gengraphrand (\omega)$, and $0$ otherwise. 
Conversely, for an element $\omega\in\{0,1\}^{\EE^d}$, let 
\begin{equation}\label{eq:Gomega}
 G_\omega = \{\edge~:~\omega_\edge =1 \} .\
\end{equation}
(By our convention, we also denote by $G_\omega$ the graph with the vertex set $\Z^d$ and the edge set $\{\edge~:~\omega_\edge = 1\}$.)
An event $A\in \mathcal F_1$ is called increasing, if for any $\omega\in A$, all the elements $\omega'$ with $G_{\omega'}\supseteq G_\omega$ are in $A$. 
The event $A$ is called decreasing, if $A^c$ is increasing. 
Throughout the text, we write $c$ and $C$ for small positive and large finite constants, respectively, that may depend on $d$ and $u$. 
Their values may change from place to place.

\subsection{Strong connectivity property}
\noindent

The following strong connectivity lemma follows from Proposition~1 in \cite{RS:Transience}. 
\begin{lemma}\label{l:connectivity}
Let $d\geq 3$, $u>0$, and $\varepsilon>0$. There exist constants $c = c(d,u,\varepsilon)>0$ and $C=C(d,u,\varepsilon)<\infty$ such 
that for all $R\geq 1$, 
\[
\mathbb P\left[\bigcap_{x,y\in \mathcal I^u\cap [0,R)^d} \left\{x\leftrightarrow y\quad\mbox{in}\quad
\intgraph^u\cap [-\varepsilon R,(1+\varepsilon)R)^d\right\}\right] 
\geq 1 - C\exp(-cR^{1/6}) .\
\]
\end{lemma}
Lemma~\ref{l:connectivity} may seem more general than Proposition~1 in \cite{RS:Transience}, 
but, in fact, the two results are equivalent. 
In order to see this, the reader may check how Proposition~1 is derived from Lemma~13 in \cite{RS:Transience}.

\subsection{Decoupling inequalities}
\noindent

Let  
\begin{equation}\label{eq:ld}
\ld = 30\cdot 4^d.\
\end{equation}
(The choice of $\ld$ will be justified in the proof of Lemma~\ref{l:badpaths}.)
Let $L_0$ and $l_0 \geq \ld$ be positive integers. We introduce the geometrically increasing sequence of length scales
\[
 L_n = l_0^n L_0, \qquad n \geq 1.
\]
 For $n\geq 0$, we introduce the renormalized lattice graph $\GG_n$ by
\[
 \GG_n = L_n \Z^d = \{L_nx ~:~ x\in\Z^d\} .\
\]
For $x\in\GG_n$ and $n\geq 0$, let 
\[
 \Lambda_{x,n} = \GG_{n-1}\cap(x+[0,L_n)^d) .\
\]

Let $\Psi_{\edge}, \edge \in\EE^d$ denote the canonical coordinates on $\{0,1\}^{\EE^d}$. 
For $x\in\GG_0$, let $\badseed_x = \badseed_{x,0} = \badseed_{x,0,L_0}$ be a $\sigma(\Psi_{\edge}, \edge \in x + [-L_0, 3L_0)^d)$-measurable event. 
We call events of the form $\badseed_{x,0,L_0}$ \emph{seed events}.
We denote the family of events $(\badseed_{x,0,L_0}~:~L_0\geq 1,x\in\GG_0)$ by $\badseed$.
Examples of seed events important for this paper will be considered in Section~\ref{subsection:seed}.
The reader should think about the events $\badseed_{x,0,L_0}$ as ``bad'' events.
Now we recursively define bad events on higher length scales using seed events. 
For $n\geq 1$ and $x\in \Z^d$, denote by $\badseed_{x,n}=\badseed_{x,n,L_0}$ the event that there exist 
$x_1,x_2\in \Lambda_{x,n}$ with $|x_1-x_2|_\infty > L_n/\ld$ such that the
events $\badseed_{x_1,n-1}$ and $\badseed_{x_2,n-1}$ occur:
\begin{equation}\label{recursive_def_A_x_n}
 \badseed_{x,n}= \bigcup_{  x_1,x_2\in \Lambda_{x,n} ; \, |x_1-x_2|_\infty > \frac{L_n}{\ld}  }
\badseed_{x_1,n-1} \cap \badseed_{x_2,n-1}   \quad.\
\end{equation}
(For simplicity, we omit the dependence of $\badseed_{x,n}$ on $L_0$ from the notation.)
Note that $\badseed_{x,n}$ is $\sigma(\Psi_{\edge}, \edge \in x + [-L_n,3L_n)^d)$-measurable. 
(This can be shown by induction on $n$.) 

Recall the definition \eqref{eq:eventsubset}. 
The following theorem is a special case of Theorem~3.4 in \cite{Sznitman:Decoupling} (modulo some minor changes that we explain in the proof). 
\begin{theorem}\label{thm:decoupling}
For all $d \geq 3$, $u > 0$ and $\delta\in(0,1)$, there exists $C = C(d,u,\delta)<\infty$ such that 
for all $n\geq 0$, $L_0\geq 1$, and $l_0\geq C$ a multiple of $\ld$,
we have 
\begin{enumerate}
 \item 
if $\badseed_x$ are decreasing events, then for all $u'\geq (1+\delta)u$, 
\begin{equation}\label{sznit_bound_decreasing}
 \mathbb P\left[\badseed_{0,n}(\intgraph^{u'})\right] 
\leq \left(l_0^{2d}
\sup_{x\in \GG_0\cap[0,L_n)^d}
\mathbb P\left[\badseed_x(\intgraph^u)\right] + \frac14\right)^{2^n} ,\
\end{equation}
 \item 
if $\badseed_x$ are increasing events, then for all $u'\leq (1-\delta)u$, 
\begin{equation}\label{sznit_bound_increasing}
 \mathbb P\left[\badseed_{0,n}(\intgraph^{u'})\right] 
\leq \left(l_0^{2d}
\sup_{x\in \GG_0\cap[0,L_n)^d}
\mathbb P\left[\badseed_x(\intgraph^u)\right] + \frac14\right)^{2^n} .\
\end{equation}
\end{enumerate}
\end{theorem}

\begin{proof}[Proof of Theorem~\ref{thm:decoupling}]
We refer the reader to Section 3 of \cite{Sznitman:Decoupling} for the notation. 
Our events $\badseed_{x,n}$ correspond to the events $G_{x,L_n}$ of \cite{Sznitman:Decoupling}, 
$\Lambda_{x,n}$ plays the role of $\Lambda$, thus $c(\mathcal G, l) = 1$ and $\lambda = d$ in  Definition~3.1 of \cite{Sznitman:Decoupling}.
There are a number of comments we would like to make before applying results derived in Section 3 of \cite{Sznitman:Decoupling}: 

(1) Even though the events $G_{x,L_n}$ in \cite{Sznitman:Decoupling} pertain to the occupancy of vertices (i.e., 
they are subsets of $\{0,1\}^{\Z^d}$), 
Theorem~3.4 in \cite{Sznitman:Decoupling} also applies in the setting when
the events $G_{x,L_n}$ pertain to the occupancy of edges (i.e., they are subsets of $\{0,1\}^{\EE^d}$),
see Theorem 2.1, Remark 2.5(3) and Corollary 2.1' of \cite{Sznitman:Decoupling}.

(2) The constant $\ld$ is taken to be $100$ in Definition~3.1 in \cite{Sznitman:Decoupling}, 
but Theorem~3.4 in \cite{Sznitman:Decoupling} works for any large enough constant $\ld$, with $l_0>\ld$ also large enough. 

(3) The events $\badseed_{x,n}$ defined by \eqref{recursive_def_A_x_n} are 
not cascading in the sense of Definition~3.1 in \cite{Sznitman:Decoupling}, 
because $(3.4)$ of \cite{Sznitman:Decoupling} only holds for $l=l_0$ rather than for all $l$ which is a multiple of $100$.
Nevertheless, the statement and the proof of Theorem~3.4 in \cite{Sznitman:Decoupling} only involve events 
$G_{x,L_n}$, with $L_n = l_0^n L_0$ for some previously fixed $L_0\geq 1$ and $l_0$ (where  $l_0$ is large enough). 

Taking the above remarks into account, we can apply Theorem~3.4 of \cite{Sznitman:Decoupling} to the events $\badseed_{x,n}$. 
In order to derive  \eqref{sznit_bound_decreasing} and \eqref{sznit_bound_increasing}
 from  Theorem~3.4 of \cite{Sznitman:Decoupling},  we choose $l_0$ large enough, 
 so that 
$u_\infty^+ \leq (1+\delta)u$, $u_\infty^- \geq (1-\delta)u$, and $l_0^{2d}\varepsilon(u_\infty^-) \leq 1/4$. 
(See, e.g., the calculations in (3.37) of \cite{Sznitman:Decoupling}.)
\end{proof}

\begin{remark}
Currently, Theorem~3.4 in \cite{Sznitman:Decoupling} (and, as a result, Theorem~\ref{thm:decoupling} of this paper) 
is proved only for increasing and decreasing events. 
It would be interesting to show that the result of Theorem~3.4 in \cite{Sznitman:Decoupling} holds for a more general class of events. 
\end{remark}

\begin{corollary}\label{cor:decoupling}
Let $d \geq 3$, $u>0$ and $\delta\in(0,1)$. 
Let $\badseed_x$ be all increasing events and $u' = (1-\delta)u$, or all decreasing events and $u' = (1+\delta)u$. 
If 
\begin{equation}\label{eq:cor:decoupling1}
\liminf_{L_0\to\infty} \sup_{x\in \GG_0\cap[0,L_n)^d} \mathbb P\left[\badseed_x(\intgraph^u)\right] = 0 ,\
\end{equation}
then there exist $l_0, L_0\geq 1$ such that for all $n\geq 0$,  
\begin{equation}\label{eq:cor:decoupling2}
\mathbb P\left[\badseed_{0,n}(\intgraph^{u'})\right] \leq 2^{-2^n} .\
\end{equation}
Moreover, if the limit in \eqref{eq:cor:decoupling1} (as $L_0\to\infty$) exists and equals to $0$, then 
there exists $C = C(d,u,\delta)<\infty$ such that the inequality \eqref{eq:cor:decoupling2} holds for all 
$l_0\geq C$ a multiple of $l(d)$, $L_0\geq C'(d,u,\delta,l_0,\badseed)$ (for some constant $C'(d,u,\delta,l_0,\badseed)$), and $n\geq 0$. 
\end{corollary}

\section{Seed events}\label{subsection:seed}
\noindent
In this section we  apply Corollary~\ref{cor:decoupling} to two families of (decreasing and increasing) bad events defined 
in terms of $\intgraph^{u}$. We also recursively define a similar (but simpler) family of bad events in terms of $\berbond^p$ and derive results analogous
to Corollary~\ref{cor:decoupling} for this family given that $p$ is close enough to $1$.
The corresponding seed events will be used in Section~\ref{sec:connectivity} to define good vertices in $\GG_0$. 
The good vertices will have the property that the existence of an infinite path of good vertices in $\GG_0$ implies the existence of 
an infinite path in the graph $\intgraph^u\cap\berbond^p$, as stated formally in Lemma~\ref{l:fromG0toZd}.

\medskip

We define the density of the interlacement at level $u$ (see, e.g., (1.58) in \cite{SznitmanAM}) by
\[
 m(u) = \mathbb P(0\in\mathcal I^u)
 = 1 - e^{-u/g(0)} ,\
\]
where $g$ is the Green function of the simple random walk on $\Z^d$ started at $0$. 
The function $m$ is continuous. 

Note that $x \in\mathcal I^u$ if and only if $\{x,y\} \in \intgraph^u$ for some $y \in \Z^d$, thus
 $\mathcal I^u$ is a measurable function of $\intgraph^u$.
It follows from Theorem 2.1 and Remark 2.2(4) of \cite{SznitmanAM}  
that $\intgraph^u$ is a translation invariant ergodic random subset of $\EE^d$.
By an appropriate ergodic theorem (see, e.g., Theorem~VIII.6.9 in \cite{DunfordSchwartz}), we get  
\begin{equation}\label{ergodicity}
\lim_{L \to \infty} \frac{1}{L^d} \sum_{ x \in  [0,L)^d} \mathds{1}\left( \, \exists y \in [0,L)^d\,:\; \{x,y\} \in \intgraph^u \, \right)
    \; \stackrel{\mathbb{P}\text{-a.s.}}{=} \; m(u).
\end{equation}

\subsection{Bad decreasing events}
\noindent

In this subsection we define and study a family of bad decreasing $\sigma(\Psi_\edge,~ \edge\in x+[0,2L_n)^d)$-measurable events 
$\overline E_{x,n}^u$ with (see \eqref{recursive_def_A_x_n})
\[
 \overline E_{x,n}^u
 =
 \bigcup_{  x_1,x_2\in \Lambda_{x,n} ; \, |x_1-x_2|_\infty > \frac{L_n}{\ld}  }
  \overline E_{x_1,n-1}^u \cap \overline E_{x_2,n-1}^u   \quad,\
\]
for $n\geq 1$, and $\mathbb P\left[\overline E_{0,n}^u(\intgraph^u)\right] \leq 2^{-2^n}$. 
In order to define the bad decreasing seed event $\overline E_x^u = \overline E_{x,0}^u$, 
we define its complement, the ``good'' increasing event $E_x^u = (\overline E_x^u)^c$.

\begin{definition}\label{def:Exu}
Fix $u>0$. Recall the definition of the graph $G_\omega$ in \eqref{eq:Gomega}. 
Let $E_x^u$ be the measurable subset of $\{0,1\}^{\EE^d}$ such that $\omega\in E_x^u$ iff
\begin{enumerate}[(a)]
\item \label{def_Exu_a} for all $e\in\{0,1\}^d$, the graph $G_\omega \cap(x + eL_0 + [0,L_0)^d)$ contains a connected component with at least $\frac34 m(u) L_0^d$ vertices, 
\item \label{def_Exu_b}  all of these $2^d$ components are connected in the graph $G_\omega \cap(x+[0,2L_0)^d)$. 
\end{enumerate}
\end{definition}
Note that $E_x^u$ is an increasing $\sigma(\Psi_\edge,~\edge\in x+[0,2L_0)^d)$-measurable event. 
Moreover, if $\gengraphrand(\omega)$ is a random translation invariant subset of $\EE^d$, then 
$\mathbb P[E_x^u(\gengraphrand)] = \mathbb P[E_0^u(\gengraphrand)]$ for all $x \in \Z^d$. 
\begin{lemma}\label{l:Ex}
For any $u>0$ there exists $\delta>0$ such that 
\begin{equation}\label{l:Ex_eq}
 \mathbb P[E_0^u(\intgraph^{u/(1+\delta)})] \to 1 ,\quad \mbox{as} \quad L_0\to\infty .\
\end{equation}
\end{lemma}
\begin{proof}[Proof of Lemma~\ref{l:Ex}] 
Let $u>0$. By the continuity of $m(u)$, we can choose $\varepsilon>0$ and $\delta>0$ so that 
\[
(1-4\varepsilon)^dm\left(\frac{u}{1+\delta}\right) > \frac34 m(u) .\
\] 
With such a choice of $\varepsilon$ and $\delta$, for $L_0\geq 1$, we obtain 
\begin{equation}\label{eq:epsilondelta}
 m\left(\frac{u}{1+\delta}\right) (L_0 - 4\lfloor \varepsilon L_0\rfloor)^d >
 \frac34 m(u) L_0^d .\
\end{equation}
Let $u' = u/(1+\delta)$. 
We consider the boxes \[B_{e}=eL_0 + [\,2\lfloor \varepsilon L_0\rfloor,L_0-2\lfloor\varepsilon L_0\rfloor\,)^d, \qquad e\in\{0,1\}^d.\]
The volume of $B_{e}$ is $|B_{e}|=(L_0 - 4\lfloor \varepsilon L_0\rfloor)^d$. 
Using \eqref{ergodicity} and \eqref{eq:epsilondelta},  we get that 
with probability tending to $1$ as $L_0\to\infty$, each of the boxes $B_{e}$, $e\in\{0,1\}^d$ contains at least
$\frac34 m(u)L_0^d$ vertices of $\mathcal I^{u'}$.

Now by Lemma~\ref{l:connectivity}, all the vertices of 
$\mathcal I^{u'}\cap B_e$ are  connected in 
$\intgraph^{u'}\cap (eL_0 + [\lfloor \varepsilon L_0\rfloor,L_0-\lfloor\varepsilon L_0\rfloor)^d)$
for all $e\in\{0,1\}^d$ with probability tending to $1$ as $L_0\to\infty$.
This shows that the event in Definition \ref{def:Exu} \eqref{def_Exu_a}  holds with probability tending to $1$ as $L_0\to\infty$.

Again by Lemma~\ref{l:connectivity}, the vertices of
$\mathcal I^{u'}\cap (eL_0 + [\lfloor \varepsilon L_0\rfloor,L_0-\lfloor\varepsilon L_0\rfloor)^d)$, $e\in\{0,1\}^d$
   are all connected in 
$\intgraph^{u'}\cap [0,2L_0)^d$. 
This, together with the previous conclusion, implies that the event in Definition \ref{def:Exu} \eqref{def_Exu_b}
 holds with probability tending to $1$ as $L_0\to\infty$.
 Hence we have established \eqref{l:Ex_eq}.
\end{proof}

\begin{corollary}\label{cor:Ex}
For each $u>0$, there exists $C = C(d,u)<\infty$ such that for all integers $l_0\geq C$ a multiple of $l(d)$ 
(see \eqref{eq:ld}), $L_0\geq C'(d,u,l_0)$ (for some constant $C'(d,u,l_0)$), and $n\geq 0$, 
\[
\mathbb P\left[\overline E_{0,n}^u(\intgraph^u)\right] \leq 2^{-2^n} .\
\]
\end{corollary}
\begin{proof}
Indeed, it immediately follows from Corollary~\ref{cor:decoupling} and Lemma~\ref{l:Ex}.  
\end{proof}

\subsection{Bad increasing events}
\noindent

In this subsection we define and study a family of bad increasing $\sigma(\Psi_\edge,~ \edge\in x+[0,2L_n)^d)$-measurable events 
$\overline F_{x,n}^u$ with (see \eqref{recursive_def_A_x_n})
\[
 \overline F_{x,n}^u
 =
 \bigcup_{  x_1,x_2\in \Lambda_{x,n} ; \, |x_1-x_2|_\infty > \frac{L_n}{\ld}  }
  \overline F_{x_1,n-1}^u \cap \overline F_{x_2,n-1}^u   \quad,\
\]
for $n\geq 1$, and $\mathbb P\left[\overline F_{0,n}^u(\intgraph^u)\right] \leq 2^{-2^n}$. 
In order to define the bad increasing seed event $\overline F_x^u = \overline F_{x,0}^u$, 
we define its complement, the ``good'' decreasing event $F_x^u = (\overline F_x^u)^c$.
\begin{definition}\label{def:Fxu}
Let $u>0$. 
Let $F_x^u$ be the measurable subset of $\{0,1\}^{\EE^d}$ such that $\omega\in F_x^u$ iff
for all $e\in\{0,1\}^d$, the graph $G_\omega \cap (x + eL_0 + [0,L_0)^d)$ contains at most 
$\frac 54 m(u) L_0^d$ vertices in connected components of size at least $2$, i.e.,
\begin{equation}\label{def_eq_no_isolated}
  \sum_{ y \in \, x + eL_0 + [0,L_0)^d} \mathds{1}\left( \, \exists z \in \,x + eL_0 + [0,L_0)^d   \,:\; \{y,z\} \in  G_\omega \, \right)
\leq \frac 54 m(u) L_0^d.
\end{equation}
\end{definition}
Note that $F_x^u$ is a decreasing $\sigma(\Psi_\edge,~\edge\in x+[0,2L_0)^d)$-measurable event.
Moreover, if $\mathcal \gengraphrand(\omega)$ is a random translation invariant subset of $\EE^d$, then 
$\mathbb P[F_x^u(\gengraphrand)] = \mathbb P[F_0^u(\gengraphrand)]$.

\begin{lemma}\label{l:Fx}
For any $u>0$ there exists $\delta\in(0,1)$ such that 
\begin{equation}\label{l:Fx_eq}
 \mathbb P[F_0^u(\intgraph^{u/(1-\delta)})] \to 1 ,\quad \mbox{as} \quad L_0\to\infty .\
\end{equation}
\end{lemma}
\begin{proof}[Proof of Lemma~\ref{l:Fx}] 
Let $u>0$. By the continuity of $m(u)$, we can choose $\delta>0$ so that 
\[
m\left(\frac{u}{1-\delta}\right) < \frac{5}{4} m(u) .\
\] 
Therefore, \eqref{ergodicity} implies that, with probability tending to $1$ as $L_0\to\infty$, 
the inequality \eqref{def_eq_no_isolated} with $G_\omega$ replaced by $\intgraph^{u/(1-\delta)}$ is satisfied for all $e\in\{0,1\}^d$. 
This implies \eqref{l:Fx_eq}.
\end{proof}

\begin{corollary}\label{cor:Fx}
For each $u>0$, there exists $C = C(d,u)<\infty$ such that for all integers $l_0\geq C$ a multiple of $l(d)$ 
(see \eqref{eq:ld}), $L_0\geq C'(d,u,l_0)$ (for some constant $C'(d,u,l_0)$), and $n\geq 0$, 
\[
\mathbb P\left[\overline F_{0,n}^u(\intgraph^u)\right] \leq 2^{-2^n} .\
\]
\end{corollary}
\begin{proof}
Indeed, it immediately follows from Corollary~\ref{cor:decoupling} and Lemma~\ref{l:Fx}. 
\end{proof}

\subsection{Bad Bernoulli events}
\noindent

In this subsection we define and study a family of bad decreasing $\sigma(\Psi_\edge,~ \edge\in x+[0,2L_n)^d)$-measurable events 
$\overline D_{x,n}$ in the spirit of the definition \eqref{recursive_def_A_x_n}: 
\[
 \overline D_{x,n}
 =
 \bigcup_{  x_1,x_2\in \Lambda_{x,n} ; \, |x_1-x_2|_\infty > \frac{L_n}{\ld}  }
  \overline D_{x_1,n-1} \cap \overline D_{x_2,n-1}   \quad,\
\]
for $n\geq 1$, and $\mathbb P\left[\overline D_{0,n}(\berbond^p)\right] \leq 2^{-2^n}$ when $p<1$ is close enough to $1$. 
We define the bad decreasing seed event $\overline D_x = \overline D_{x,0}$ as 
the measurable subset of $\{0,1\}^{\EE^d}$ such that $\omega\in \overline D_x$ iff 
there is an edge in the box $x+[0,2L_0)^d$ which is not in $G_\omega$
(remember that an edge $\edge$ is in $x+[m,n)^d$ if both its endvertices are in $x+[m,n)^d$), i.e.,
\begin{equation}\label{def_eq_Dx}
\overline D_x= \left\{\omega\in\{0,1\}^{\EE^d}~:~ (x+[0,2L_0)^d) \cap \EE^d \nsubseteq G_\omega \right\}.
\end{equation}
Note that $\overline D_x$ is a decreasing $\sigma(\Psi_\edge,~\edge\in x+[0,2L_0)^d)$-measurable event.
Moreover, if $\mathcal \gengraphrand(\omega)$ is a random translation invariant subset of $\EE^d$, then 
$\mathbb P[\overline D_x(\gengraphrand)] = \mathbb P[\overline D_0(\gengraphrand)]$.

\begin{lemma}\label{l:Hx}
For any integers $L_0\geq 1$ and $l_0> 2 l(d)$ there exists $p<1$ such that for all $n\geq 0$, 
\[
\mathbb P \left[\overline D_{0,n}(\berbond^p)\right] \leq 2^{-2^n} .\
\]
\end{lemma}
\begin{proof}[Proof of Lemma~\ref{l:Hx}]
Since the probability of $\overline D_0(\berbond^p)$ is at most $1-p^{d (2L_0)^d}$, we can choose $p = p(L_0,l_0) < 1$ so that 
\[
 l_0^{2d} \mathbb P \left[\overline D_0(\berbond^p)\right] < 1/2 .\
\]
Note that for $x_1,x_2\in \GG_{n-1}$, $|x_1- x_2|_\infty\geq L_n/\ld$, 
the events $\overline D_{x_1,n-1}(\berbond^p)$ and $\overline D_{x_2,n-1}(\berbond^p)$ are independent and have the same probability. 
Therefore, since $|\Lambda_{x,n}| \leq l_0^d$, we get
\[
\mathbb P \left[\overline D_{0,n}(\berbond^p)\right] 
\leq 
l_0^{2d} \mathbb P \left[\overline D_{0,n-1}(\berbond^p)\right]^2 
\leq \dots \leq
\left(l_0^{2d}\right)^{1+2+\ldots+2^{n-1}}\left(\mathbb P \left[\overline D_0(\berbond^p)\right]\right)^{2^n}
\leq
\left(l_0^{2d} \mathbb P \left[\overline D_0(\berbond^p)\right]\right)^{2^n} .\
\]
The result follows from the choice of $p$. 
\end{proof}

\section{Connected components of bad boxes are small}\label{sec:connectivity}
\noindent

For $x,y\in \GG_0$, we say that $x$ and $y$ are nearest-neighbors in $\GG_0$ if $|x-y|_1 = L_0$, 
and $*$-neighbors in $\GG_0$ if $|x-y|_\infty = L_0$. 
We say that $\pi = (x(1),\ldots, x(m)) \subset \GG_0$ is a nearest-neighbor path in $\GG_0$, if 
for all $j$, $x(j)$ and $x(j+1)$ are nearest-neighbors in $\GG_0$, and 
a $*$-path in $\GG_0$, if for all $j$, $x(j)$ and $x(j+1)$ are $*$-neighbors in $\GG_0$. 

Let $u>0$ and $p\in (0,1)$. Recall the definitions of the bad seed events $\overline E_x^u=(E_x^u)^c$,
$\overline F_x^u=(F_x^u)^c$ and $\overline D_x$ from Definition \ref{def:Exu}, Definition \ref{def:Fxu} and \eqref{def_eq_Dx}, respectively.
 We say that $x\in \GG_0$ is a {\it bad} vertex if the event
\[
\overline D_x(\berbond^p) \cup \overline E_x^u(\intgraph^u) \cup \overline F_x^u(\intgraph^u)
\]
occurs. Otherwise, we say that $x$ is {\it good}.
The following lemma will be useful in the proofs of Theorems~\ref{thm:slabs} and \ref{thm:transience}. 
\begin{lemma}\label{l:fromG0toZd}
Let $x$ and $y$ be nearest-neighbors in $\GG_0$, and assume that they are both good. 

(a) Each of the graphs $(\intgraph^u\cap\berbond^p)\cap(z + [0,L_0)^d)$, with $z\in \{x,y\}$, 
contains the unique connected component $\mathcal C_z$ with at least $\frac34 m(u) L_0^d$ vertices, and

(b) $\mathcal C_x$ and $\mathcal C_y$ are connected in the graph $(\intgraph^u\cap\berbond^p)\cap ((x+[0,2L_0)^d)\cup(y+[0,2L_0)^d))$. \\
In particular, this implies that if there is an infinite nearest-neighbor path $\pi = (x_1,\ldots)$ of good vertices in $\GG_0$, then the 
set $\cup_{i=1}^\infty (x_i + [0,2L_0)^d)$ contains an infinite nearest-neighbor path of $\intgraph^u\cap\berbond^p$. 
\end{lemma}
\begin{proof}
Let $x$ and $y$ be nearest-neighbors in $\GG_0$, and assume that they are both good. By Definition~\ref{def:Exu},  
the graphs $\intgraph^u\cap(x + [0,L_0)^d)$ and $\intgraph^u\cap(y + [0,L_0)^d)$ contain 
connected components of size at least $\frac 34 m(u)L_0^d$, which are connected in the graph 
$\intgraph^u\cap ((x+[0,2L_0)^d)\cup(y+[0,2L_0)^d))$. 

By Definition~\ref{def:Fxu}, 
each of the graphs $\intgraph^u\cap(x + [0,L_0)^d)$ and $\intgraph^u\cap(y + [0,L_0)^d)$ 
contains at most $\frac 54 m(u) L_0^d$ vertices in connected components of size at least $2$. 
Since $2\cdot \frac 34 > \frac 54$, there can be at most one connected component of size $\geq \frac 34 m(u)L_0^d$ 
in each of the graphs $\intgraph^u\cap(x + [0,L_0)^d)$ and $\intgraph^u\cap(y + [0,L_0)^d)$. 
This impies that each of the graphs $\intgraph^u\cap(z + [0,L_0)^d)$, with $z\in\{x,y\}$, 
contains the unique connected component $\mathcal C_z$ with at least $\frac 34 m(u) L_0^d$ vertices, and 
$\mathcal C_x$ and $\mathcal C_y$ are connected in the graph 
$\intgraph^u\cap ((x+[0,2L_0)^d)\cup(y+[0,2L_0)^d))$. 

Finally, by \eqref{def_eq_Dx}, $((x+[0,2L_0)^d)\cup(y+[0,2L_0)^d))\subseteq \berbond^p$. 
Therefore, all the edges of the graph $\intgraph^u\cap ((x+[0,2L_0)^d)\cup(y+[0,2L_0)^d))$ are present in $\berbond^p$. 
\end{proof}

For $x\in \GG_0$, and $M<N$ which are divisible by $L_0$, 
let $\overline{H}^*(x,M,N)$ be the event that $B(x,M)$ is connected to the boundary of $B(x,N)$ by a $*$-path of bad vertices in $\GG_0$. 
Let $\overline{H}^*(x,N) = \overline{H}^*(x,0,N)$ be the event that $x$ is connected to the boundary of $B(x,N)$ by a $*$-path of bad vertices in $\GG_0$.

\begin{lemma}\label{l:badpaths}
For any $u>0$, there exist $L_0\geq 1$, $p<1$, $c>0$ and $C<\infty$ (all depending on $u$) 
such that for all $N$ divisible by $L_0$, we have 
\begin{equation}\label{eq:badpaths}
 \mathbb P[\overline{H}^*(0,N)] \leq C e^{- N^c} .\
\end{equation}
\end{lemma}
\begin{proof}[Proof of Lemma~\ref{l:badpaths}]
We may assume that $N\geq 2L_0$. 
It suffices to show that for $n\geq 0$, 
\begin{equation}\label{eq:badcrossings}
\mathbb P[\overline{H}^*(0,L_n,2L_n)] \leq C e^{- L_n^c} .\
\end{equation}
Indeed, choose $n$ so that $2L_n \leq N < 2L_{n+1} = 2l_0L_n$. Then 
\[
 \mathbb P[\overline{H}^*(0,N)]\leq \mathbb P[\overline{H}^*(0,L_n,2L_n)] \leq C e^{-L_n^c} \leq C'e^{-N^{c'}} .\
\]
Let $u>0$. Choose $l_0 (> \ld)$, $L_0\geq 1$, and $p<1$ such that Corollaries \ref{cor:Ex} and \ref{cor:Fx} and Lemma~\ref{l:Hx} hold. 
For $n\geq 0$ and $x\in \GG_n$, we say that $x$ is $n$-{\it bad} if the event 
\[
\overline D_{x,n}(\berbond^p) \cup \overline E_{x,n}^u(\intgraph^u) \cup \overline F_{x,n}^u(\intgraph^u)
\]
occurs. Otherwise, we say that $x$ is $n$-{\it good}. (In particular, $x$ is $0$-bad if and only if $x$ is bad.)
By the definition of 
$\overline D_{x,n}(\berbond^p)$, $\overline E_{x,n}^u(\intgraph^u)$ and $\overline F_{x,n}^u(\intgraph^u)$, 
\begin{equation}\label{eq:ngood}
\begin{split}
\text{if} &\text{ $x\in\GG_n$ is $n$-good, then there exist at most three $(n-1)$-bad vertices}\\
&\text{$z_1,\ldots,z_s\in\GG_{n-1}\cap(x+[0,L_n)^d)$ (with $0\leq s\leq 3$) such that $|z_i-z_j|_\infty > L_n/\ld$ for all $i\neq j$.}
\end{split}
\end{equation}
In order to prove \eqref{eq:badcrossings}, it suffices to show that for all $n\geq 0$ and $x\in\GG_n$, 
\begin{equation}\label{eq:inclusion}
 \overline{H}^*(x,L_n,2L_n) \subseteq 
\bigcup_{y \in \GG_n\cap(x+[-2L_n,2L_n)^d)} 
\{y\mbox{ is } n\mbox{-bad}\} .\
\end{equation}
Indeed, since the number of vertices in $\GG_n\cap [-2L_n,2L_n)^d = \{-2L_n,-L_n,0,L_n\}^d$ equals $4^d$, 
we obtain by translation invariance that 
\[
\mathbb P[\overline{H}^*(0,L_n,2L_n)] 
\leq 4^d \left(\mathbb P[\overline D_{0,n}(\berbond^p)] + \mathbb P[\overline E_{0,n}^u(\intgraph^u)] + 
\mathbb P[\overline F_{0,n}^u(\intgraph^u)]\right)
\leq 4^d\cdot 3\cdot 2^{-2^n} 
\leq C e^{-L_n^c} .\
\]
We prove \eqref{eq:inclusion} by induction on $n$. The statement is obvious for $n=0$. 
We assume that \eqref{eq:inclusion} holds for all integers smaller than $n\geq 1$, and will show that it also holds for $n$. 
It suffices to prove the induction step for $x = 0$. 
The proof goes by contradiction. 
Assume that $\overline{H}^*(0,L_n,2L_n)$ occurs and all the vertices in $\{-2L_n,-L_n,0,L_n\}^d$ are $n$-good. 
Let $\pi$ be a $*$-path of bad vertices in $\GG_0$ from $B(0,L_n)$ to the boundary of $B(0,2L_n)$. 
Let $m_0 = \lfloor l_0/5\rfloor - 1$. 
Note that the path $\pi$ intersects the boundary of each of the boxes $B(0,L_n + 5 L_{n-1} i)$, for $i\in\{0,\ldots,m_0\}$. 
Therefore, there exist $y_0,\ldots,y_{m_0}\in\GG_{n-1}$ such that for all $i\in\{0,\ldots,m_0\}$, 
(a) $|y_i|_\infty = L_n + 5 L_{n-1} i$ and 
(b) $\pi\cap B(y_i,L_{n-1})\neq \emptyset$ 
(see Figure~\ref{fig:yi}).
\begin{figure}
\begin{center}
\vskip-1cm

\includegraphics[scale=0.3]{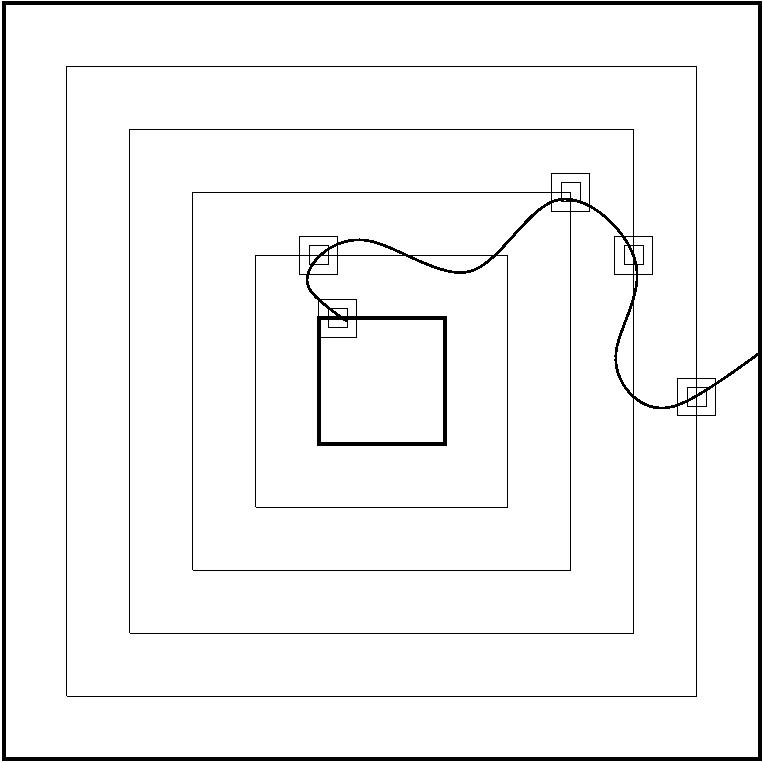}
\caption{One way to define $y_i$ is as the closest vertex in $\GG_{n-1}\cap \partial B(0,L_n+5L_{n-1}i)$ to 
the point of the first intersection of $\pi$ with $\partial B(0,L_n+5L_{n-1}i)$. 
Concentric boxes are not drawn to scale here: the innermost box is $B(0,L_n)$, the outermost box is $B(0,2L_n)$, 
and the intermediate boxes are $B(0,L_n + 5 L_{n-1} i)$, for $i\in\{1,\ldots,m_0\}$. 
The smallest and second smallest boxes along the path $\pi$ are $B(y_i,L_{n-1})$ and $B(y_i,2L_{n-1})$, respectively, for $i\in\{0,\ldots,m_0\}$.}
\label{fig:yi}
\end{center}
\end{figure}
By the definition of $m_0$ and $y_i$'s, all the boxes $B(y_i,2L_{n-1})$ are disjoint and contained in $[-2L_n,2L_n)^d$, and 
the path $\pi$ connects $B(y_i,L_{n-1})$ to the boundary of $B(y_i,2L_{n-1})$, i.e., 
the event $\overline{H}^*(y_i,L_{n-1},2L_{n-1})$ occurs for all $i\in\{0,\ldots,m_0\}$. 
We will show that 
\begin{equation}\label{eq:badpaths:induction}
\begin{split}
\text{there exists $j$ such that all the $4^d$ vertices in $\GG_{n-1}\cap(y_j + [-2L_{n-1},2L_{n-1})^d)$
are $(n-1)$-good,}
\end{split}
\end{equation} 
which will contradict our assumption that \eqref{eq:inclusion} holds for $n-1$. 

Since all the vertices in $\GG_n\cap [-2L_n,2L_n)^d$ are $n$-good by assumption, it follows from \eqref{eq:ngood} that  
\begin{equation}\label{eq:badboxes}
\begin{split}
\text{there exist } &\text{$z_1,\ldots,z_{3\cdot 4^d}\in [-2L_n,2L_n)^d$ such that}\\
&\text{all the vertices in $(\GG_{n-1}\cap[-2L_n,2L_n)^d)\setminus \cup_{i=1}^{3\cdot 4^d}B(z_i,L_n/\ld)$ are $(n-1)$-good.}
\end{split}
\end{equation}

Note that each of the balls $B(z,2L_n/\ld)$ contains at most $(4(L_n/\ld)+1)/(5 L_{n-1}) \leq l_0/\ld$ different $y_i$'s. 
Therefore, the union of the balls $\cup_{i=1}^{3\cdot 4^d}B(z_i,2L_n/\ld)$ (with $z_i$'s defined in \eqref{eq:badboxes}) 
contains at most $3\cdot 4^d \cdot l_0/\ld$ different $y_i$'s, 
which is strictly smaller than $m_0$ by the choice of $\ld$ in \eqref{eq:ld}. 
We conclude that there exists $j\in \{0,\ldots,m_0\}$ such that 
\[
y_j\notin \cup_{i=1}^{3\cdot 4^d}B(z_i,2L_n/\ld) .\
\]
We assume that $l_0$ is chosen large enough so that $L_n/\ld > 2L_{n-1}$, i.e., $l_0 > 2\ld$. 
With this choice of $l_0$, 
\begin{equation}\label{eq:ballyj}
B(y_j,2L_{n-1})\subseteq [-2L_n,2L_n)^d\setminus \cup_{i=1}^{3\cdot 4^d}B(z_i,L_n/\ld) .\
\end{equation}
Therefore, \eqref{eq:badpaths:induction} follows from \eqref{eq:badboxes} and \eqref{eq:ballyj}, 
which is in contradiction with the assumption 
that \eqref{eq:inclusion} holds for $n-1$. 
This implies that \eqref{eq:inclusion} holds for all $n\geq 0$. 
The proof of Lemma~\ref{l:badpaths} is completed.
\end{proof}

\section{Proofs of Theorem~\ref{thm:slabs} and Theorem~\ref{thm:transience}}\label{sec:proofs}
\noindent

In this section, we derive Theorems~\ref{thm:slabs} and \ref{thm:transience} 
from Lemmas~\ref{l:fromG0toZd} and \ref{l:badpaths}.

\begin{proof}[Proof of Theorem~\ref{thm:slabs}]
The two results of Theorem~\ref{thm:slabs} can be proved similarly 
(note that the results of Sections~\ref{section:notationandproperties}-\ref{sec:connectivity} can be 
trivially adapted to site percolation on $\mathcal I^u$), 
therefore we only provide a proof for the case of bond percolation on $\intgraph^u$.

Choose $L_0$ and $p<1$ such that Lemma~\ref{l:badpaths} holds.  
Remember the definitions of a bad vertex and the event $\overline{H}^*(0,N)$ from Section~\ref{sec:connectivity}. 
Let $M$ be a positive integer. 
Note that the probability that there exists a $*$-circuit of bad vertices in $\GG_0\cap (\Z^2\times\{0\}^{d-2})$ around $[0,L_0M)^2\times\{0\}^{d-2}$ is at most 
\[
 \sum_{N=M}^\infty \mathbb P[\overline{H}^*(0,L_0N)] \leq C\sum_{N=M}^\infty e^{-N^c} \leq 1/2 ,\
\]
for large enough $M$. If there is no such circuit, then, by planar duality (see, e.g., \cite[Chapter~3.1]{Grimmett}), 
there is a nearest-neighbor path $\pi = (x_0,x_1,\ldots)$
of good vertices in $\GG_0\cap (\Z^2\times\{0\}^{d-2})$ that connects $[0,L_0M)^2\times\{0\}^{d-2}$ to infinity. 
Namely, for all $i$, $x_i\in \GG_0\cap (\Z^2\times\{0\}^{d-2})$, 
$|x_i - x_{i+1}|_1 = L_0$, $x_i$ is good, $x_0\in [0,L_0M)^2\times\{0\}^{d-2}$, 
and $|x_n|_\infty\to\infty$ as $n\to\infty$. 
It follows from Lemma~\ref{l:fromG0toZd} that the graph 
$\{\edge~:~\edge\in x + [0,2L_0)^d~\mbox{for some } x\in \pi\} \subset \Z^2\times [0,2L_0)^{d-2}$ 
contains an infinite connected component of $\intgraph^u\cap \berbond^p$. 
Therefore, the probability that an infinite nearest-neighbor path in $\intgraph^u\cap \berbond^p$ visits 
$[0,L_0M + 2 L_0)^2\times[0,2L_0)^{d-2}$ is at least $1/2$. 
By the ergodicity of $\intgraph^u\cap \berbond^p$, 
an infinite nearest-neighbor path in $(\intgraph^u\cap \berbond^p)\cap(\Z^2\times [0,2L_0)^{d-2})$ exists with probability $1$. 
\end{proof}

\begin{proof}[Proof of Theorem~\ref{thm:transience}]
We will use the main result of \cite{PP:Transience} that for an infinite graph $G=(V,E)$ and i.i.d. positive random variables $R_\edge$, $\edge\in E$, 
the following statements are equivalent: (a) almost surely, the electric network $\{R_\edge~:~\edge\in E\}$ is transient, and 
(b) for some $p<1$, independent bond percolation on $G$ with parameter $p$ contains with positive probability a cluster 
on which simple random walk is transient. (In the proof, we will only use the easy implication, namely, that (b) implies (a).)

Therefore, in order to prove Theorem~\ref{thm:transience}, it suffices to show that for some $p<1$, with positive probability, the graph 
$\intgraph^u \cap \berbond^p$ contains a transient subgraph. 
The proof of this fact is similar to the proof of Theorem~1 in \cite{RS:Transience}, so we only give a sketch here.  

Let $d\geq 3$ and $u>0$. 
Denote by $S^d$ the $d$-dimensional Euclidean unit sphere.  
We will show that, for any $\varepsilon\in (0,1)$, 
there exists an event $\mathcal H$ of probability $1$ such that if $\mathcal H$ occurs, then 
\begin{equation}\label{eq:percinparaboloids}
\begin{split}
\text{the graph $\intgraph^u \cap \berbond^p$} &\text{ contains an infinite connected subgraph}\\
&\text{which, for each $v\in S^d$, contains an infinite path in the set $\cup_{n=1}^{\infty} B(nv, 2n^{\varepsilon} )$.}
\end{split}
\end{equation}
(The set $\bigcup_{n=1}^{\infty} B(nv, 2n^{\varepsilon} )$ is roughly shaped like a paraboloid with an axis parallel to $v$.)
After that, one can proceed, as in Section~3 of \cite{RS:Transience}, 
to show that this infinite connected subgraph of $\intgraph^u \cap \berbond^p$ is transient. 

Remember the definitions of the bad vertex and the event $\overline{H}^*(x,N)$
from Section~\ref{sec:connectivity}.
Let $L_0$ and $p<1$ satisfy Lemma~\ref{l:badpaths}.
By \eqref{eq:badpaths} and the Borel-Cantelli lemma, for any $\varepsilon \in (0,1)$, the following event $\mathcal H$ has probability $1$:  
there exists a (random) $m$ such that for all $x\in \GG_0$ with $|x|_\infty \geq m L_0$, 
the event $\overline{H}^*(x, |x|_\infty^\varepsilon)$
does not occur. 
It remains to show that if the event $\mathcal H$ occurs, then \eqref{eq:percinparaboloids} holds. 

We will first prove that the event $\mathcal H$ implies that 

(a) for each $v\in S^d$, there is a nearest-neighbor path $\pi_v$ of good vertices in 
$\GG_0\cap \cup_{n=1}^{\infty} B(nv, n^{\varepsilon} )$ that connects $B(0,m L_0)$ to infinity, and 

(b) all the paths $\pi_v$ are connected by nearest-neighbor paths of good vertices in $\GG_0\cap B(0,2 m L_0)$. \\
Indeed, assume first that (a) fails, i.e., there exists $v\in S^d$ such that 
the set of vertices $y\in \GG_0\cap \cup_{n=1}^{\infty} B(nv, n^{\varepsilon} )$
connected to $B(0,m L_0)$ by a nearest-neighbor path of good vertices in $\GG_0\cap\cup_{n=1}^{\infty} B(nv, n^{\varepsilon} )$ 
is finite. 
By \cite[Lemma~2.1]{DP96} or \cite[Theorem~3]{timar}, the boundary of this set contains a $*$-connected subset $\mathcal S$
of bad vertices in $\GG_0\cap\cup_{n=1}^{\infty} B(nv, n^{\varepsilon} )$ 
such that any nearest-neighbor path from $B(0,m L_0)$ to infinity in $\GG_0\cap\cup_{n=1}^{\infty} B(nv, n^{\varepsilon} )$ intersects $\mathcal S$. 
In particular, there exists $x\in \GG_0$ with $|x|_\infty \geq m L_0$, such that 
the event $\overline{H}^*(x, |x|_\infty^\varepsilon)$ occurs; 
and, therefore, the event $\mathcal H$ does not occur. 

Similarly, if (a) holds and (b) fails, then there exist at least two disjoint connected components of good vertices 
of diameter $\geq m L_0$ in $\GG_0\cap (B(0, 2 m L_0)\setminus B(0, m L_0-1))$ that intersect $B(0, m L_0)$. 
Therefore, by \cite[Lemma~2.1]{DP96} or \cite[Theorem~3]{timar}, 
there exists $x\in \GG_0$ with $|x|_\infty = m L_0$ such that the event $\overline{H}^*(x, m L_0)$ occurs. 
This again implies that the event $\mathcal H$ does not occur. 

It remains to notice that by (a), (b) and Lemma~\ref{l:fromG0toZd}, the occurence of $\mathcal H$ implies  
\eqref{eq:percinparaboloids}. 
Indeed, 
Lemma~\ref{l:fromG0toZd} and (a) imply that there is an infinite 
path of $\intgraph^u \cap \berbond^p$ in every set $\bigcup_{n=1}^{\infty} B(nv, 2n^{\varepsilon} )$, $v\in S^d$,  
and Lemma~\ref{l:fromG0toZd} and (b) imply that 
all these infinite paths are in the same connected subgraph of $\intgraph^u \cap \berbond^p$. 

Therefore, we have constructed the event $\mathcal H$ of probability $1$ which implies \eqref{eq:percinparaboloids}. 
In order to show that the infinite cluster in \eqref{eq:percinparaboloids} is transient, 
we proceed identically to the proof of Theorem~1 in Section~3 of \cite{RS:Transience}. 
We omit further details. 
\end{proof}

\section{Proof of Theorem~\ref{thm:Vuvarepsilon}}\label{sec:vacantset}

In this section we prove Theorem~\ref{thm:Vuvarepsilon}. 
The first statement of Theorem~\ref{thm:Vuvarepsilon} is proved in Section~\ref{sec:thm:Vuvarepsilon:existence}. 
In Section~\ref{sec:def:overlineuL} we state Theorem~\ref{thm:u*varepsilon}, which 
implies the second statement of Theorem~\ref{thm:Vuvarepsilon}. 
The result of Theorem~\ref{thm:u*varepsilon} is more general than the one of Theorem~\ref{thm:Vuvarepsilon}, 
since it also provides explicit upper and lower bounds on $u_*(\varepsilon)$, as $\varepsilon\to 0$. 
We prove Theorem~\ref{thm:u*varepsilon} in Section~\ref{sec:proof:u*varepsilon}.

\subsection{Existence of phase transition}\label{sec:thm:Vuvarepsilon:existence}

In this section we prove the first statement of Theorem~\ref{thm:Vuvarepsilon}. It follows from the next lemma. 
\begin{lemma}\label{l:coupling}
For any $0<u<u'$ and $\varepsilon\in(0,1/2)$, 
the set $\mathcal V^{u',\varepsilon}$ is stochastically dominated by $\mathcal V^{u,\varepsilon}$. 
In particular, for any $u> u_*(\varepsilon)$, almost surely, 
the set $\mathcal V^{u,\varepsilon}$ does not contain an infinite connected component. 
\end{lemma}
\begin{proof}
Note that by the construction of $(\mathcal I^u)_{u>0}$, on the same probability space in \cite[(1.53)]{SznitmanAM}, 
the set $\mathcal I^{u}$ is stochastically dominated by $\mathcal I^{u'}$ for $u<u'$. 

Let $\varepsilon\in(0,1/2)$. 
Let $\xi_x$, $x\in\Z^d$, be independent Bernoulli random variables with parameter $\varepsilon$, and 
$\eta_x$, $x\in\Z^d$, independent Bernoulli random variables with parameter $(1-2\varepsilon)/(1-\varepsilon)$, 
the two families are mutually independent, and also independent from the random interlacement $\mathcal I^u$. 
Let $\varphi_x = \max(\xi_x,\eta_x\mathds{1}(x\in\mathcal I^u))$. It is easy to see that 
given $\mathcal I^u$, the $\varphi_x$ are independent, 
and the probability that $\varphi_x = 1$ equals $\varepsilon$ for $x\in \mathcal V^u$, and $(1-\varepsilon)$ for $x\in\mathcal I^u$. 
Therefore, the set of vertices $\{x\in\Z^d~:~\varphi_x = 1\}$ has the same distribution as $\mathcal I^{u,\varepsilon}$. 
Since, for $u<u'$, $\mathcal I^u$ is stochastically dominated by $\mathcal I^{u'}$, 
we deduce that $\mathcal I^{u,\varepsilon}$ is stochastically dominated by $\mathcal I^{u',\varepsilon}$, 
and, therefore, $\mathcal V^{u',\varepsilon}$ is stochastically dominated by $\mathcal V^{u,\varepsilon}$. 
\end{proof}

\subsection{Phase transition is non-trivial}\label{sec:def:overlineuL}

In this section
we state that for small enough $\varepsilon>0$, $u_*(\varepsilon)\in(0,\infty)$ and give 
explicit upper and lower bounds on $u_*(\varepsilon)$, as $\varepsilon \to 0$.
The main result of this section is Theorem~\ref{thm:u*varepsilon}, which will be proved in Section~\ref{sec:proof:u*varepsilon}. 
In order to state the theorem, we need to define the critical thresholds $\overline u$ and $u_{**}$. 

\begin{remark}
The earlier version of this paper contained a different proof of the fact that 
$u_*(\varepsilon)\in(0,\infty)$. 
It was based on a new notion of the so-called strong supercriticality in slabs. 
That proof is available in the first version of this paper on the arXiv \cite{RS:Disordered:v1}. 
The proof we present here is significantly simpler and relies on recent local uniqueness results of \cite{DRS}. 
\end{remark}

\begin{definition}\label{def:tildeu}
Let $d\geq 3$. Let $\overline u = \overline u(d)$ be the supremum over all $u'$ such that for each $u$ smaller than $u'$, 
there exist constants $c = c(d,u)>0$ and $C = C(d,u)<\infty$ such that for all $n\geq 1$, 
we have 
\begin{equation}\label{eq:tildeu:1}
\mathbb P\left[
B(0,n) \leftrightarrow \infty \mbox{ in } \mathcal V^u
\right]
\geq 
1 - C e^{-n^c} ,\
\end{equation}
and 
\begin{equation}\label{eq:tildeu:2}
\mathbb P\left[
\begin{array}{c}
\text{any two connected subsets of $\mathcal V^u\cap B(0,n)$ with}\\
\text{diameter $\geq n/10$ are connected in $\mathcal V^u\cap B(0,2n)$}
\end{array}
\right]
\geq 1 - C e^{-n^c} .\
\end{equation}
\end{definition}
Note that Definition~\ref{def:tildeu} implicitly implies that the right hand side of \eqref{eq:tildeu:1} 
must be positive for all $u<\overline u$ and large enough $n$. In particular, we conclude that $\overline u \leq u_* < \infty$. 
It was recently proved in \cite[Theorem~1.1]{DRS} (and, for $d\geq 5$, earlier in \cite[(1.2) and (1.3)]{Teixeira}) that 
\begin{equation}\label{eq:tildeu:DRS}
\overline u > 0 \mbox{ for all } d\geq 3 .\
\end{equation}
Let us also recall the definition of $u_{**}$ from \cite[(0.6)]{Sznitman:Upperbound} and \cite[(0.10)]{Sznitman:Decoupling}: 
\begin{equation}\label{def:u**}
u_{**}= \inf \left\{ u \geq 0 \; : \; \liminf_{L \to \infty} \mathbb{P}\left[ 
\begin{array}{c}
\text{$B(0,L)$ is connected to the boundary of $B(0,2L)$}\\
\text{by a nearest-neighbor path in $\mathcal V^u$}
\end{array}
\right]
=0 \right\} .\
\end{equation}
It follows from \cite{Sznitman_Sidoravicius_cpam,SznitmanAM,Sznitman:Decoupling} 
that
\[
u_* \leq u_{**} < \infty \mbox{ for all } d\geq 3 .\
\] 
We prove the following theorem.
\begin{theorem}\label{thm:u*varepsilon}
Let $d\geq 3$. We have
\begin{equation}\label{eq:u*varepsilon}
0<\overline u \leq \liminf_{\varepsilon \to 0}  u_*(\varepsilon) 
\leq 
\limsup_{\varepsilon \to 0}  u_*(\varepsilon) \leq u_{**} <\infty .\
\end{equation}
\end{theorem}
\begin{remark}\label{rem:conjecture:u*varepsilon}
It would be interesting to understand whether the phase transition of $\mathcal V^u$ is actually 
stable with respect to small random noise. 
In other words, is it true that 
\begin{equation}\label{eq:conjecture:u*varepsilon}
\lim_{\varepsilon\to 0} u_*(\varepsilon) = u_* ?
\end{equation}
Based on \eqref{eq:u*varepsilon}, an affirmative answer to \eqref{eq:conjecture:u*varepsilon} will be obtained  
as soon as one proves that 
\begin{equation}\label{eq:relations:u*}
\overline u = u_{*}=u_{**} .\
\end{equation}
Note that the thresholds $\overline u$ and $u_{**}$ are defined purely in terms of $\mathcal V^u$, 
and not $\mathcal V^{u,\varepsilon}$. 
The statement \eqref{eq:relations:u*} is about local connectivity properties of sub- and supercritical phases of $\mathcal V^u$. 
In the context of Bernoulli percolation, similar thresholds can be defined, and 
it is known that they coincide with the threshold for the existence of an infinite component 
(see, e.g., \cite[(5.4) and (7.89)]{Grimmett}), i.e., the analogue of \eqref{eq:relations:u*} holds. 
The main challenge in proving \eqref{eq:relations:u*} comes from the long-range dependence in $\mathcal V^u$ and 
the lack of the so-called BK-inequality (see, e.g., \cite[(2.12)]{Grimmett}), and hence it is interesting in its own. 
\end{remark}

\subsection{Proof of Theorem~\ref{thm:u*varepsilon}}\label{sec:proof:u*varepsilon}

Recall the definition of $\mathcal B^\varepsilon$ from the beginning of Section~\ref{sec:mainresults}. 
In order to prove \eqref{eq:u*varepsilon}, it suffices to show that 
\begin{align}
\label{stable_supercritical}
\forall\, u<\overline u \quad \exists \, \varepsilon_0(u)>0  \quad \forall \, 
\varepsilon < \varepsilon_0(u)\, : \quad 
&\mathbb{P} [ 0 \stackrel{\mathcal{V}^u \setminus \mathcal{B}^{\varepsilon}}{\longleftrightarrow} 
\infty ]>0,\quad \mbox{and} \\
\label{stable_subcritical}
\forall\, u>u_{**} \quad \exists \, \varepsilon_0(u)>0  \quad \forall \, 
\varepsilon < \varepsilon_0(u) \, : \quad &\mathbb{P} [ 0 \stackrel{\mathcal{V}^u \cup \mathcal{B}^{\varepsilon}}{\longleftrightarrow} 
\infty ]=0 .\
 \end{align}
The proofs of these statements are very similar to the proof of Theorem~\ref{thm:slabs}. 
Therefore, we only sketch the main ideas here. 

We begin with the proof of \eqref{stable_supercritical}. 
Let 
\begin{equation}\label{def:etau}
\eta(u) = \mathbb P\left[0\leftrightarrow \infty\mbox{ in }\mathcal V^u \right] .\
\end{equation}
Note that $u_{*} = \inf\{u\geq 0~:~\eta(u) = 0\}$. 
It follows from \cite[Corollary~1.2]{Teixeira_AAP} that 
\begin{equation}\label{eq:etacont}
\text{$\eta(u)$ is continuous on $[0,u^*)$.}
\end{equation}
\begin{definition}\label{def:vacantk}
For $u>0$ and $k\geq 0$, 
let $\mathcal V^u_k$ be the subset of vertices of $\mathcal V^u$ which are 
in connected components of diameter $\geq k$ in $\mathcal V^u$. 
\end{definition}
By \eqref{def:etau} and Definition~\ref{def:vacantk}, 
$\mathbb P[0\in \mathcal V^u_k]\geq \eta(u)$ and $\mathbb P[0\in\mathcal V^u_k] \to \eta(u)$ as $k\to\infty$.
Therefore, by an appropriate ergodic theorem (see, e.g., \cite[Theorem~VIII.6.9]{DunfordSchwartz} and \cite[Theorem~2.1]{SznitmanAM}), 
we get  
\begin{equation}\label{eq:etaergodicity}
\lim_{L \to \infty} \frac{1}{L^d} 
\sum_{ x \in [0,L)^d} \mathds{1}\left( \,  x \in\mathcal V^u_L \, \right)
    \; \stackrel{\mathbb{P}\text{-a.s.}}{=} \; 
\lim_{L \to \infty} \frac{1}{L^d} 
\sum_{ x \in [0,L)^d} \mathds{1}\left( \,  x \leftrightarrow \infty\mbox{ in }\mathcal V^u \, \right)
    \; \stackrel{\mathbb{P}\text{-a.s.}}{=} \; \eta(u) .\
\end{equation}

\begin{definition}\label{def:good}
Let $u<\overline u$ and $L_0\geq 1$. 
We call $x\in \GG_0$ a good vertex if the following conditions are satisfied: 
\begin{itemize}
\item[(i)]
for all $e\in\{0,1\}^d$, the graph $\mathcal V^u_{L_0} \cap(x + eL_0 + [0,L_0)^d)$ contains 
a connected component with at least $\frac34 \eta(u) L_0^d$ vertices, and 
all these $2^d$ components are connected in $\mathcal V^u \cap(x+[0,2L_0)^d)$,
\item[(ii)]
for all $e\in\{0,1\}^d$, $|\mathcal V^u_{L_0} \cap (x + eL_0 + [0,L_0)^d)| \leq \frac 54 \eta(u) L_0^d$, 
\item[(iii)]
$(x+[0,2L_0)^d) \cap \mathcal B^\varepsilon = \emptyset$.
\end{itemize}
Otherwise we call $x$ a bad vertex. 
Note that the event $\{\text{$x$ is good}\}$ is measurable with respect to the 
$\sigma$-algebra 
generated by $\{\mathds{1}(y\in\mathcal V^u)~:~ y\in x+ [-L_0, 3L_0)^d\}$ and 
$\{\mathds{1}(z\in\mathcal B^\varepsilon)~:~z\in x+[0,2L_0)^d\}$. 
\end{definition} 
Definition~\ref{def:good} is similar to the definition of a good vertex in Section~\ref{sec:connectivity}, 
except that now we are dealing with $\mathcal V^u_{L_0}$, rather than with $\intgraph^u$. 
In particular, the event $\{\text{$x$ is good}\}$ pertains to the occupancy of the vertices of $\Z^d$ rather than the edges. 
The event in (i) corresponds to the event $E^u_x(\intgraph^u)$, 
the event in (ii) corresponds to the event $F^u_x(\intgraph^u)$, and the event in (iii) corresponds to 
the complement of the event $\overline D_x(\berbond^p)$. 
The role of the continuous function $m(u)$ in Definitions~\ref{def:Exu} and \ref{def:Fxu} 
is played by $\eta(u)$ (see \eqref{eq:etacont} and compare \eqref{eq:etaergodicity} to \eqref{ergodicity}).
The role of Lemma~\ref{l:connectivity} is played by the following lemma.
\begin{lemma}\label{l:connectivity:bigdiam}
Let $d\geq 3$, $0<u<\overline u$, and $\varepsilon>0$. 
There exist constants $c = c(d,u,\varepsilon)>0$ and $C=C(d,u,\varepsilon)<\infty$ such 
that for all $R\geq 1$, 
\begin{equation}\label{eq:connectivity:bigdiam}
\mathbb P\left[
\bigcap_{x,y\in \mathcal V^u_{\varepsilon R}\cap [0,R)^d} \left\{x\leftrightarrow y\quad\mbox{in}\quad
\mathcal V^u\cap [-\varepsilon R,(1+\varepsilon)R)^d\right\}\right] 
\geq 1 - C e^{-R^{c}} .\
\end{equation}
\end{lemma}
\begin{proof}[Proof of Lemma~\ref{l:connectivity:bigdiam}]
It suffices to consider $R\geq 1$ such that $\varepsilon R\geq 10$. 
Let $k = \lfloor \varepsilon R/10\rfloor$. 
For $z\in [0,R)^d$, let $\mathcal A_z$ be the event that 
\begin{itemize}
\item[(a)]
$B(z,k)$ is connected to the boundary of $B(z,4k)$ in $\mathcal V^u$, and 
\item[(b)]
every two nearest-neighbor paths from $B(z,2k)$ to the boundary of $B(z,3k)$ in $\mathcal V^u$ 
are in the same connected component of $\mathcal V^u\cap B(z,6k)$. 
\end{itemize}
Let $\mathcal A = \cap_{z\in [0,R)^d}\mathcal A_z$. 
By \eqref{eq:tildeu:1} and \eqref{eq:tildeu:2}, 
there exist constants $\widetilde c = \widetilde c(d,u,\varepsilon)>0$ and $\widetilde C=\widetilde C(d,u,\varepsilon)<\infty$, 
such that for all $R$, we have 
\[
\mathbb P\left[\mathcal A\right]\geq 1 - \widetilde C e^{-R^{\widetilde c}} .\
\]
Therefore, it suffices to show that 
\begin{equation}\label{eq:Aimpliesconnect}
\text{the event $\mathcal A$ implies the event in \eqref{eq:connectivity:bigdiam}.} 
\end{equation}
Let $x,y\in \mathcal V^u_{\varepsilon R}\cap[0,R)^d$.
Let $\mathcal C_x$ and $\mathcal C_y$ be the connected components of $x$ and $y$ 
in $\mathcal V^u\cap[-\varepsilon R,(1+\varepsilon)R)^d$.
We will show that if $\mathcal A$ occurs then $\mathcal C_x = \mathcal C_y$. 
Note that by the choice of $x,y$ and $k$, $\mathcal C_x$ contains a path from $x$ to the boundary of $B(x,4k)$, 
and $\mathcal C_y$ contains a path from $y$ to the boundary of $B(y,4k)$.

Assume that $\mathcal A$ occurs. 
Take a nearest-neighbor path $\pi = (z_1,\ldots,z_t)$ in $[0,R)^d$ from $x$ to $y$. 
For each $1\leq i\leq t-1$, the occurrence of the events $\mathcal A_{z_i}$ and $\mathcal A_{z_{i+1}}$ implies that 
(a) there exist nearest-neighbor paths $\pi_1$ and $\pi_2$ in $\mathcal V^u$, 
$\pi_1$ from $B(z_i,k)$ to the boundary of $B(z_i,4k)$, and 
$\pi_2$ from $B(z_{i+1},k)$ to the boundary of $B(z_{i+1},4k)$, and 
(since both paths connect $B(z_i,2k)$ to the boundary of $B(z_i,3k)$)
(b) any two such paths are connected in $\mathcal V^u\cap B(z_i,6k)$. 
This implies that $\mathcal C_x$ and $\mathcal C_y$ must be connected in 
$\mathcal V^u\cap \cup_{i=1}^t B(z_i,6k) \subseteq \mathcal V^u\cap[-\varepsilon R,(1+\varepsilon)R)^d$. 
This finishes the proof of \eqref{eq:Aimpliesconnect} and of the lemma. 
\end{proof}

\bigskip

Using \eqref{eq:etacont}, \eqref{eq:etaergodicity}, and Lemma~\ref{l:connectivity:bigdiam}, 
we can proceed similarly to the proof of \eqref{eq:badpaths}
(see also the proofs of Corollaries \ref{cor:Ex} and \ref{cor:Fx} and Lemma~\ref{l:Hx}) to show that
for any $0<u<\overline u$, there exist $L_0\geq 1$, $c>0$ and $C<\infty$ such that for 
all $N$ divisible by $L_0$, we have 
\begin{equation}\label{eq:badpaths:vacantset}
\begin{split}
\mathbb P\left[
\begin{array}{c}
\text{$0$ is connected to the boundary of $B(0,N)$}\\
\text{by a $*$-path of bad vertices in $\GG_0$}
\end{array}
\right] 
\leq C e^{-N^c} .\
\end{split}
\end{equation}
We now use planar duality, similarly to the proof of Theorem~\ref{thm:slabs}, to show that \eqref{eq:badpaths:vacantset} implies
that for large enough $L_0$, 
\begin{equation}\label{eq:vacantset:infcluster}
\mathbb P\left[
\begin{array}{c}
\text{$0$ is connected to infinity}\\
\text{by a nearest-neighbor path of good vertices in $\GG_0$}
\end{array}
\right]>0 .\
\end{equation} 
Similarly to Lemma~\ref{l:fromG0toZd}, we observe that
if there exists an infinite nearest-neighbor path $\pi=(x_1,\ldots)$ of good vertices in $\GG_0$, 
then the set $\cup_{i=1}^\infty\left(x_i + [0,2L_0)^d\right)$ 
contains an infinite nearest-neighbor path of $\mathcal V^u\setminus \mathcal B^\varepsilon$.
This, together with \eqref{eq:vacantset:infcluster}, implies \eqref{stable_supercritical}.

\bigskip

We proceed with the proof of \eqref{stable_subcritical}. 
Let $u>u_{**}$, $L_0\geq 1$, and $\varepsilon \in (0, 1/L_0^{d+1})$. 
Recall that $\GG_0 = L_0 \Z^d$. 
We call $x\in \GG_0$ a bad vertex if either

(a) there exists a nearest-neighbor path in $\mathcal V^u$ from $B(x,L_0)$ to the boundary of $B(x,2L_0)$, \\ or 

(b) $\mathcal B^{\varepsilon}\cap B(x,2L_0)\neq \emptyset$.\\
With the above choice of $\varepsilon$, the probability of event in (b) goes to $0$ as $L_0\to\infty$. 

It follows from the definition of $u_{**}$ and the choice of $\varepsilon$ 
(similarly to the proof of \eqref{eq:badpaths}) that 
for any $u>u_{**}$, there exist $L_0\geq 1$, $c>0$ and $C<\infty$ such that for   
all $N$ divisible by $L_0$, we have 
\[
\begin{split}
\mathbb P\left[\text{$0$ is connected to the boundary of $B(0,N)$ by a $*$-path of bad vertices in $\GG_0$}\right] 
\leq C e^{-N^c} .\
\end{split}
\]
In particular, for any $u>u_{**}$ and large enough $L_0$, 
almost surely, there is no infinite nearest-neighbor cluster of bad vertices in $\GG_0$.
Finally, note that if $\pi$ is an infinite path in $\mathcal V^{u}\cup\mathcal B^{\varepsilon}$ from the origin, then 
the origin is in an infinite nearest-neighbor path of bad vertices in $\GG_0$. 
This implies \eqref{stable_subcritical}.
\qed

\paragraph{Acknowledgements.}
We thank A.-S. Sznitman for pointing out a connection between our results and  
questions of robustness to noise, which led to results discussed in Section~\ref{sec:vacantset}, 
and for a careful reading of the manuscript.

\end{document}